\documentclass[11pt,leqno]{amsart}
\usepackage{amsmath,amsfonts,amssymb,amscd,amsthm,amsbsy,upref}
\usepackage{color}
 \usepackage{comment}
 \usepackage{amsrefs}

\newtheorem{thm}{Theorem}[section]
\newtheorem{lem}[thm]{Lemma}
\newtheorem{cor}[thm]{Corollary}
\newtheorem{prop}[thm]{Proposition}

\theoremstyle{definition}
\newtheorem{defin}[thm]{Definition}

\newtheorem{rem}[thm]{Remark}

\newtheorem{remark}[thm]{Remark}
\newtheorem{remarks}[thm]{Remarks}

\newtheorem{thmAlfa}{Theorem}

\def\H{{\mathbb H}}
\def\J{{\mathbb J}}

\def\N{{\mathbb N}}
\def\M{{\mathbb M}}
\def\L{{\mathbb L}}
\def\R{{\mathbb R}}
\def\Z{{\mathbb Z}}

\def\cT{{\mathcal T}}

\def\cE{{\mathcal E}}
\def\cF{{\mathcal F}}
\def\cC{{\mathcal C}}

\def\cG{{\mathcal G}}

\newcommand{\keq}{\!=\!}
 
\newcommand{\kleq}{\!\leq\!}

\newcommand{\kgr}{\!>\!}
\newcommand\kin{\!\in\!}
\newcommand{\ksubset}{\!\subset\!}

\newcommand{\kplus}{\!+\!}

\newcommand{\lb}{\bar l}
\newcommand{\nb}{\bar n}
\newcommand{\mb}{\bar m}
\newcommand{\ub}{\bar u}
\newcommand{\vb}{\bar v}

\def\vp{\varepsilon}

\newcommand{\xt}{{\tilde x}}

\newcommand{\cof}{\text{\rm cof}}
\newcommand{\spa}{\text{\rm span}}
\newcommand{\supp}{\text{\rm supp}}
\newcommand{\ie}{{\it i.e.,\ }}
\newcommand{\co}{\mathrm{c}_0}
\newcommand{\diam}{\text{\rm diam}}
\newcommand{\Lip}{\text{\rm Lip}}

\begin{document}

\allowdisplaybreaks

\title{A new coarsely rigid class of Banach spaces}

\author{F.~Baudier}
\address{F.~Baudier, Department of Mathematics, Texas A\&M University, College Station, TX 77843, USA}
\email{florent@math.tamu.edu}

\author{G.~Lancien}
\address{G.~Lancien, Laboratoire de Math\'ematiques de Besan\c con, Universit\'e Bourgogne Franche-Comt\'e, 16 route de Gray, 25030 Besan\c con C\'edex, Besan\c con, France}
\email{gilles.lancien@univ-fcomte.fr}

\author{P.~Motakis}
\address{P.~Motakis, Department of Mathematics, University of
Illinois at Urbana-Champaign, Urbana, IL 61801, U.S.A.}
\email{pmotakis@illinois.edu}

\author{Th.~Schlumprecht}
\address{Th.~Schlumprecht, Department of Mathematics, Texas A\&M University, College Station, TX 77843-3368, USA, and Faculty of Electrical Engineering,
Czech Technical University in Prague, Zikova 4, 16627, Prague, Czech Republic}
\email{schlump@math.tamu.edu}

\thanks{The first named author was supported by the National Science
Foundation under Grant Number DMS-1800322.
The second named author was supported by the French
``Investissements d'Avenir'' program, project ISITE-BFC (contract
 ANR-15-IDEX-03).
The third named author was  supported by the National Science Foundation
under Grant Numbers DMS-1600600 and DMS-1912897.
The fourth named author was supported by the National Science Foundation under Grant Numbers DMS-1464713 and DMS-1711076 .}

\begin{abstract} We prove that the class of reflexive asymptotic-$c_0$
Banach spaces is coarsely rigid, meaning that if a Banach space $X$
coarsely embeds into a reflexive asymptotic-$c_0$ space $Y$, then $X$
is also reflexive and asymptotic-$c_0$. In order to achieve this result
we provide a purely metric characterization of this class of Banach spaces. This metric characterization takes the form of a concentration inequality for Lipschitz maps on the Hamming
graphs, which is rigid under coarse embeddings. Using an example of a quasi-reflexive asymptotic-$c_0$ space,
we show that this concentration inequality is not equivalent to the
non equi-coarse embeddability of the Hamming graphs.
\end{abstract}

\maketitle

 \setcounter{tocdepth}{1}
\tableofcontents
\section{Introduction}\label{S:1}

The concept of rigidity for a class of mathematical objects has permeated
mathematical fields. A prime example of a rigidity problem arose in geometric
group theory. Take a finitely generated group $\Gamma$ which is an algebraic
object. One can apprehend $\Gamma$ in the category of metric spaces by looking
at its Cayley graph. Then, a fundamental aspect of
Gromov's geometric group theory program \cite{Gromov1984} is to understand how
much of the algebraic properties of a group one can recover knowing solely the
large scale shape of its Cayley graph. A class $\cG$ of groups is said to be
quasi-isometrically rigid if every group that is quasi-isometric to a group in
$\cG$ is actually virtually isomorphic to a group in $\cG$. A quasi-isometric
embedding is what we call a coarse-Lipschitz embedding in this paper
(see all the relevant definitions of non linear embeddings in
subsection \ref{SS:2.3}). It is quite remarkable that many classes of groups
are known to be quasi-isometrically rigid: free groups, hyperbolic groups,
amenable groups, and we refer to \cite{Kapovich2014} for a detailed list.

In this work, we provide a Banach space analogue of this type of results.
A class $\cC$ of Banach spaces is called {\em coarsely rigid} if it follows
from $Y$ being a member of $\cC$ and $X$ being coarsely embedded into $Y$, that
$X$ is also in $\cC$. Let us insist on the fact that coarse embeddings are very weak embeddings.
Indeed, it is classical that $L_1$ coarsely embeds into $\ell_2$ (while it does not coarse-Lipschitz embed). On the other hand Nowak \cite{Nowak2006} showed that for any $p\in [1,\infty)$, $\ell_2$ coarsely embeds into $\ell_p$. This was extended by Ostrovskii \cite{Ostrovskii2009} who proved that $\ell_2$ coarsely embeds into any Banach space with an unconditional basis and of non-trivial cotype. On a more elementary level, note that $\R$ coarse-Lipschitz embeds into $\Z$. Therefore coarsely rigid classes are rare. The class of spaces that coarsely embed into a fixed metric space $(M,d)$,
or the class of  spaces in which a fixed $(M,d)$ does not coarsely embed are
clearly coarsely rigid. It is for instance rather simple to see that a Banach
space $X$ has dimension less than $n \in \N$ if and only if the integer grid
$\Z^n$ equipped with the $\ell_1$ metric does not coarsely embed into $X$.
Besides such simple coarsely rigid classes, very few rigidity results have
been obtained so far. Let us describe three important examples. Randrianarivony
showed in \cite{Randrianarivony2006} that a quasi-Banach space $X$ coarsely
embeds into a Hilbert space if and only if there is a probability space
$(\Omega,B, \mu)$ such that $X$ is linearly isomorphic to a subspace of
$L_0(\Omega,B, \mu)$.
This clearly describes a class of quasi-Banach spaces that is coarsely rigid.
Then, a major achievement
by Mendel and Naor \cite{MendelNaor2008} was a  purely metric extension of the
linear notion of Rademacher cotype.
 Using that notion of  {\em metric cotype} they were able to show that within the
 class of Banach spaces with non-trivial type the class
 $\{X\colon \inf\{q'\ge2\colon X \textrm{ has Rademacher cotype }q'\}\le q\}$
 is coarsely rigid. It is still unclear and important to understand whether
 the non-trivial type restriction is necessary. Another important  rigidity
 result was obtained by Kalton \cite{Kalton2007}. Indeed,
 he  showed that, within the class of Banach spaces that do not have
 $\ell_1$-spreading models (or equivalently spaces with the alternating
 Banach-Saks property), the class of reflexive Banach spaces is coarsely rigid.
 It then follows from an ultraproduct argument that, within the class of Banach
 spaces with non-trivial type, the class of super-reflexive Banach spaces is
 coarsely rigid. Since $\ell_1$ coarsely embeds into $\ell_2$ we need at
 least to exclude spaces which contain $\ell_1$, to obtain both conclusions.
 The last papers of N. Kalton (\cite{Kalton2007} among others, see also the survey
 \cite{GLZ2014} and references therein) show that asymptotic structures
 of Banach spaces often provide linear properties that are invariant under coarse or
 coarse-Lipschitz embeddings. Our work follows this program: studying the links between asymptotic
 structures and large scale geometry of Banach spaces.

In this article we exhibit a new example of an unrestricted class of
infinite dimensional Banach spaces that is coarsely rigid. The notion of an
{\em asymptotic-$c_0$ space} will be recalled
in Section \ref{S:3}.


\begin{thmAlfa}\label{thmA} Let $Y$ be a reflexive asymptotic-$\co$ Banach space.
If $X$ is a Banach space that coarsely embeds into $Y$, then $X$ is also reflexive and asymptotic-$\co$.
\end{thmAlfa}

Since there are reflexive asymptotic-$\co$ spaces, like Tsirelson's original space $T^*$ \cite{Tsirelson1974}, which will be recalled later, Theorem \ref{thmA} immediately implies the main result from \cite{BaudierLancienSchlumprecht2018}, where the existence of an infinite dimensional Banach space that does not coarsely contain $\ell_2$ is proved.
Our proof of Theorem \ref{thmA}, which is carried out in Section \ref{S:4}, follows
from the following purely metric characterization of the linear property of being
``reflexive and asymptotic-$c_0$'' in terms of a concentration inequality for Lipschitz maps
on the Hamming graphs (see the definition and notation in subsection \ref{SS:2.2}).

\begin{thmAlfa}\label{thmB}
A Banach space $X$ is reflexive and asymptotic-$c_0$ if and only if there exists $C\geq 1$
such that for every $k\in\N$ and every Lipschitz map $f:\big([\N]^k,d^{(k)}_\H\big)\to X$ there exists $\M\in[\N]^\omega$ so that
\begin{equation*}
\diam\big( f([\M]^k)\big) \leq C\mathrm{Lip}(f).
\end{equation*}
\end{thmAlfa}

This concentration inequality was introduced in  \cite{BaudierLancienSchlumprecht2018} where
it was shown to hold for maps taking values into Tsirelson's original space $T^*$.
The space $T^*$ is the prototypical example of a separable reflexive asymptotic-$\co$ Banach space,
and the proof from \cite{BaudierLancienSchlumprecht2018} can be generalized to show that the same
concentration inequality holds for maps with values into any reflexive asymptotic-$\co$ Banach space.
The more conceptual approach undertaken in this article to prove that any reflexive asymptotic-$c_0$ Banach
space satisfies the above  metric concentration inequality, requires the central notion of asymptotic structure
from \cite{MaureyMilmanTomczak1995} which is described in section \ref{S:3}.
In order to prove the converse, the crucial step is to show that if a Banach space $X$ satisfies the metric
concentration inequality, then all its asymptotic models generated by weakly-null arrays are isomorphic to $\co$.
The notion of asymptotic models was introduced by Halbeisen and Odell in \cite{HalbeisenOdell2004}.
Then the conclusion follows from
an unexpected link between the notions {\em asymptotic structure} and {\em asymptotic models}
(see Section \ref{S:3}). Indeed, it  was proved by D. Freeman, E. Odell, B. Sari and B. Zheng \cite{FOSZ2017}
that a separable Banach space  which does not contain  a copy of  $\ell_1$ is asymptotic-$\co$ whenever all its
asymptotic models generated by weakly-null arrays are isomorphic to $\co$.

The concentration inequality in Theorem \ref{thmB} clearly prevents the
equi-coarse embeddability of the sequence of Hamming graphs. We show in section
\ref{quasi-reflexive} that the converse is not true.
More precisely, we provide an example of a non reflexive Banach space in which the Hamming graphs do not equi-coarsely embed.

\section{Preliminaries}\label{S:2}
\subsection{Trees}\label{SS:2.1}
For $k\kin \N$  we put $[\N]^{\le k}=\{ S\ksubset \N: |S|\kleq k\}$,   $[\N]^{ k}=\{ S\ksubset \N: |S|=k\} $, $[\N]^{<\omega}=\bigcup_{k\in \N}[\N]^{\kleq k}$,
 $[\N]^\omega\keq\{ S\ksubset \N: S\text{ infinite}\} $, and $[\N]\keq\{S:S\ksubset \N\}$. We always list the elements of some $\mb\in [\N]^{< \omega}$ or in $[\N]^{ \omega}$ in increasing order, \ie if we write $\mb=\{m_1,m_2,\ldots, m_l\}$ or  $\mb=\{m_1,m_2,m_3, \ldots \}$, we tacitly assume that $m_1<m_2<\ldots $. For $\mb=\{m_1,m_2,\ldots, m_r\}\in [\N]^{\le k}$ and $\nb=\{n_1,n_2,\ldots, n_s\}\in [\N]^{\le k}$  we write $\mb\prec \nb$, if $r<s\le k$ and  $m_i=n_i$, for $i=1,2,\ldots, r$, and we write $\mb\preceq \nb$ if $\mb\prec \nb$ or $\mb=\nb$. Note that $[\N]^{\le k}$, $k\in\N$ with $\prec$, are {\em rooted trees},  \ie partial orders with a unique minimal element, namely $\emptyset$, and the property that for each $\nb\in[\N]^{\le k}$,
the set of  {\em predecessors of $\nb$}  $\{\mb:\mb\prec \nb\}$ is finite and linearly ordered.

In this paper we will only consider trees of finite height. For a set $X$ we will call a family $(x_{\nb}:\nb\in[\N]^{\le k})$, for $k\in\N$, a {\em tree of height $k$}. Sometimes we are also considering  {\em unrooted trees of height $k$}, which are families of the form $(x_{\nb}:\nb\in[\N]^{\le k}\setminus\emptyset)$. We call  for $\nb\kin[\N]^k$ a sequence of the form $(x_{\mb}: \mb\preceq \nb)=(x_{\{n_1,n_2,\ldots, n_l\}})_{l=0}^k $ {\em a branch of the tree $(x_{\nb}:\nb\in[\N]^{\le k})$}, and $(x_{\mb}:\emptyset\prec \mb\preceq \nb)=(x_{\{n_1,n_2,\ldots, n_l\}})_{l=1}^k$ {\em a branch of the unrooted tree} $(x_{\nb}:\nb\in[\N]^{\le k}\setminus\{\emptyset\})$. Sequences of the form $(x_{\nb\cup\{ i\}})_{i>\max(\nb)}$ where $\nb\in [\N]^{\le k-1}$ (for a tree of height $k$), are called {\em nodes} of the tree $(x_{\nb}:\nb\in[\N]^{\le k})$.

If $(x_{\nb}:\nb\in[\N]^{\le k})$ is a tree in $X$ and $\M=\{m_1,m_2,\ldots \}\in[\N]^\omega $ we call  $(x_{\mb}:\mb\in[\M]^{\le k})$ a {\em refinement of  $(x_{\nb}:\nb\in[\N]^{\le k})$}.
 By  relabeling $\xt_{\nb}=x_{\{m_i:i\in\nb\}}$, for $\nb\in[\N]^{\le k}$,
  the family $(\xt_{\nb}:\nb\in[\N]^{\le k})$  is a tree which we also call a refinement of  $(x_{\nb}:\nb\in[\N]^{\le k})$.

If $X$ is a Banach space we call a tree $(x_{\nb}:\nb\in[\N]^{\le k})$ in $X$ {\em normalized} if $x_{\nb}\in S_X$, for all $\nb\in[\N]^{\le k}$,  and weakly convergent, or weakly null
  if all its nodes are weakly converging or  weakly null, respectively. Here $S_X$ denotes the unit sphere in $X$, while $B_X$ denotes the closed unit ball.

\subsection{Hamming graph on  $[\N]^k$}\label{SS:2.2}
  For  $k\in\N$ and $\mb=\{m_1,m_2,\ldots,m_ k\} $ and $\nb=\{n_1,n_2,\ldots, n_k\}$ in $ [\N]^{k}$  we define the {\em Hamming distance}, by
 \begin{equation}\label{E:2.2.1}
 d^{(k)}_\H(\mb,\nb)=\big|\{ i\in \{1,2,\ldots,k\}: m_i\not=n_i\}\big|
 \end{equation}
 and put $\H^\omega_k=\big([\N]^k, d^{(k)}_\H\big)$.

\subsection{Embeddings}\label{SS:2.3}
Let $(X,d_X)$ and $(Y,d_Y)$ be two metric spaces and $f:X\to Y$. One defines
$$\rho_f(t)=\inf\big\{d_Y(f(x),f(y)) : d_X(x,y)\geq t\big\},$$
and
$$\omega_f(t)={\rm sup}\{d_Y(f(x),f(y)) : d_X(x,y)\leq t\}.$$

\noindent Note that for every $x,y\in X$,
\begin{equation}\label{E:2.3.1}
\rho_f(d_X(x,y))\le d_Y(f(x),f(y))\le\omega_f(d_X(x,y)).
\end{equation}

The moduli $\rho_f$ and $\omega_f$ will  be called the \textit{compression modulus} and the \textit{expansion modulus} of the map $f$, respectively. We adopt the convention $\sup(\emptyset)=0$ and $\inf(\emptyset)=+\infty$. The map $f$ is a \textit{coarse embedding} if $\lim_{t\to \infty}\rho_f(t)=\infty$ and $\omega_f(t)<\infty$ for all $t>0$. A map $f\colon X\to Y$ is said to be a \textit{uniform embedding} if $\lim_{t\to 0}\omega_f(t)=0$ and $\rho_f(t)>0$ for all $t>0$, i.e. $f$ is an injective uniformly continuous map whose inverse is uniformly continuous.

If one is given a family  of metric spaces $(X_i)_{i\in I}$, one says that {\em $(X_i)_{i\in I}$ equi-coarsely (resp. equi-uniformly)} embeds into $Y$ if there exist non-decreasing functions $\rho, \omega\colon [0,\infty)\to[0,\infty)$ and for all $i\in I $, maps $f_i\colon X_i\to Y$ such that $\rho\le \rho_{f_i}$, $\omega_{f_i}\le \omega$, and $\lim_{t\to \infty}\rho(t)=\infty$ and $\omega(t)<\infty$ for all $t>0$ (resp. $\lim_{t\to 0}\omega(t)=0$ and $\rho(t)>0$ for all $t>0$).

We call a map  $f:X\to Y$ {\em Lipschitz continuous } if
$$\text{Lip}(f)=\sup\Big\{ \frac{d(f(x),f(y)}{d(x,y)}: x,y\in X, d(x,y)>0\Big\}<\infty,$$
and we call it a {\em bi-Lipschitz embedding}, if it is injective and, if $f$ and $f^{-1}$ are both Lipschitz continuous.

A {\em  coarse Lipschitz embedding} is a map $f:X\to Y$, for which there  are numbers $\theta\ge 0$, and  $0<c_1<c_2$, so that

\begin{equation}\label{E:2.3.2}
c_1d_X(x,y)\le d_Y(f(x),f(y))\le c_2d_X(x,y), \text{ whenever $x,y\in X$ and $d(x,y)\ge \theta$}.
\end{equation}

\section{Asymptotic properties of Banach spaces and their interplay}\label{S:3}

For two basic sequences $(x_i)$ and $(y_i)$ in some Banach spaces $X$ and $Y,$ respectively, and $C\ge 1$, we say that   $(x_i)$ and $(y_i)$ are {\em $C$-equivalent}, and we write $(x_i)\sim_C(y_i)$, if there are positive  numbers $A$ and $B$,
with $C=A\cdot B$,
so that for all $(a_j)\in c_{00}$, the vector space of all sequences  $x=(\xi_j)$ in $\R$ for which the support $\supp(x)=\{j: \xi_j\not=0\}$ is finite,  we have
$$\frac1{A}\Big\|\sum_{i=1}^\infty a_i x_i\Big\|\le \Big\|\sum_{i=1}^\infty a_i y_i\Big\|\le B\Big\|\sum_{i=1}^\infty a_i x_i\Big\|.$$
In that case we say that  $\frac1A$ is the {\em lower estimate} and $B$ the {\em upper estimate of $(y_i)$ with respect to $(x_i)$}.
Note that $(x_i)$ and $(y_i)$ are  $C$-equivalent if and only $C\ge \|T\|\cdot\|T^{-1}\|$, where the linear operator  $T:   \spa(x_i:i\in\N) \to  \spa(y_i:i\in\N)$, is defined by $T(x_i)=y_i$, $i\kin\N$.

If $(e_i)$ is a Schauder basis of a Banach space $X$, we recall that $(x_n)$ is a {\em block sequence} in $X$ with respect to the basis $(e_i)$ if $\max( \supp(x_1))<\min(\supp(x_2))\le \max(\supp(x_2))<\cdots \le  \max
(\supp(x_{n-1}))<\min(\supp(x_n))\le \cdots.$

For $k\in\N$ we denote by $\cE_k$ the set of all norms on $\R^k$, for which the unit vector basis $(e_i)_{i=1}^k$ is a normalized monotone basis. With an easily understood abuse of terminology this can also be referred to as the set of all pairs $(E,(e_j)_{j=1}^k)$, where $E$ is a $k$-dimensional Banach space and $(e_j)_{j=1}^k$ is a monotone basis of $E$.

We define a metric $\delta_k$ on $\cE_k$  as follows :
 For two spaces $E=(\R^k,\|\cdot\|_E)$ and $F=(\R^k,\|\cdot\|_F)$  we let $\delta_k(E,F)=\log\big( \|I_{E,F}\|\cdot \|I_{E,F}^{-1}\|\big)$,
$I_{E,F}:E\to F$,  is the formal identity. It also well known and easy to show that $(\cE_k, \delta_k)$ is a compact metric space.
The following definition is due to Maurey, Milman, and Tomczak-Jaegermann  \cite{MaureyMilmanTomczak1995}.

\begin{defin}{(The  $k$-th asymptotic structure of $X$ \cite{MaureyMilmanTomczak1995}.)}\label{D:3.1}

Let $X$ be a Banach space. For $k\in\N$ we define the {\em $k$-th asymptotic structure of $X$} to be the set, denoted by $\{X\}_k$, of spaces $E=(\R^k,\|\cdot\|)\in \cE_k$ for which the following is true:
 \begin{align}\label{E:3.1.1}
 \forall \vp\kgr0\, \forall X_1\kin\cof(X)\,&\exists x_1\kin S_{X_1}\, \forall X_2\kin\cof(X)\,\exists x_2\kin S_{X_2}\,\ldots  \forall X_k\kin\cof(X)\,\exists x_k\kin S_{X_k}\,\\
 &(x_j)_{j=1}^k\sim_{1+\vp} (e_j)_{j=1}^k.\notag
 \end{align}
 For $1\le p\le \infty$ and $c\ge 1$, we say  that $X$ {\em is $c$-asymptotically $\ell_p$}, if for  all  $k\kin \N$ and all spaces $E\in\{X\}_k$, with monotone normalized basis  $(e_j)_{j=1}^k$, $(e_j)_{j=1}^k$ is $c$-equivalent to the $\ell_p^k$ unit vector basis.
 We say  that $X$ {\em is asymptotically $\ell_p$}, if it is $c$-asymptotically $\ell_p$ for some $c\ge 1$. In case that $p=\infty$ we say that the space $X$ is $c$-asymptotically $c_0$, or asymptotically $c_0$.
\end{defin}

We denote by $T^*$ the Banach space originally constructed by Tsirelson in
 \cite{Tsirelson1974}. It was the first example of a Banach space that does not contain any isomorphic copies of $\ell_p$ nor $c_0$. Since it is the archetype of a reflexive asymptotic-$c_0$
 space, we explain shortly its construction (we will also use it at the end of section \ref{quasi-reflexive}). Soon after, in \cite{FigielJohnson1974}, it became clear that the
 more natural space to define is $T$, the dual of $T^*$, because the norm of this space is more conveniently described. It has since become common to refer to $T$ as Tsirelson space instead of $T^*$. Figiel and Johnson in \cite{FigielJohnson1974} gave an implicit formula that describes the norm of $T$ as follows. For $E,F\in [\mathbb N]^{<\omega}$  and $n\in \mathbb N$ we mean by $n\le E$ that $n\le \min E$ and by $E<F$ that $\max(E)< \min(F)$. We call a sequence $(E_j)_{j=1}^n$ of finite subsets of $\mathbb{N}$ {\em admissible } if $n\le E_1<E_2<\cdots<E_n$. For $x=\sum_{j=1}^\infty \lambda_j e_j\in c_{00}$ and $E\in [\N]^{<\omega}$ we write $Ex=\sum_{j\in E} \lambda_j e_j$. As it was observed in \cite{FigielJohnson1974}, if  $\|\cdot\|_T$ denotes the norm of $T$ then for every $x\in c_{00}$:
   \begin{equation}\label{E:6.1}
    \|x\|_T=\max\Big\{\|x\|_\infty, \frac12 \sup\sum_{j=1}^n \|E_jx\|_T\Big\},
   \end{equation}
where the supremum is taken over all $n\in\N$ and admissible sequences $(E_j)_{j=1}^n$. The space $T$ is the completion of $c_{00}$ with this norm and the unit vector basis is a 1-unconditional basis. Then it was proven in \cite{FigielJohnson1974} that $T$ does not contain a subspace isomorphic to $\ell_1$, which, together with the easy  observation that $T$ certainly does not contain a subspace isomorphic to $c_0$, yields by James' Theorem  \cite[Theorem 2]{James1950} that $T$ must be reflexive.

The following property of $T^*$ (see \cite{Tsirelson1974}*{Lemma 4}) is essential:
\begin{equation}\label{E:6.2}
\Big\| \sum_{j=1}^n x_j\Big\|_{T^*}\le 2\max_{1\leq j\leq n}\|x_j\|_{T^*} \text{ if $(x_j)_{j=1}^n$ is a block sequence, with $n\le \supp(x_1)$.}
\end{equation}
The fact that $T^*$ is $2$-asymptotic-$c_0$ is an easy consequence of the above estimate. This well known fact is hard to track down in the literature, and follows from the fact that every weakly null tree admits a refinement for which all branches are  arbitrary small perturbations of blocks.

\begin{remarks}\label{R:3.2}
Let us recall some easy facts about the asymptotic structure of a Banach space which can be found in  \cite{KnaustOdellSchlumprecht1999}, \cite{MaureyMilmanTomczak1995} or \cite{OdellSchlumprecht2002}.
\begin{enumerate}
 \item[a)]
 Let  $E=(\R^k,\|\cdot\|)$, with $\|\cdot\| $ being a norm on $\R^k$, for which $(e_j)$ is  a normalized basis (but not necessarily monotone). If $(e_j)$ satisfies  \eqref{E:3.1.1} for some infinite dimensional Banach space $X$, then $(e_j)_{j=1}^k$ is automatically a  monotone basis of $E$ (by using the ideas of Mazur's proof that normalized weakly null sequences have  basic subsequences with a basis constant which is arbitrarily close to $1$). Therefore the above introduced  definition of asymptotic structure coincides with the original one given in \cite{MaureyMilmanTomczak1995}.
\item[b)]
 For any infinite dimensional Banach space $X$ and $k\in\N$, $\{X\}_k$ is a closed and thus  compact subset of $\cE_k$ with respect to the above introduced  metric $\delta_k$ on $\cE_k$.

\item[c)]
 For a $k$-dimensional space $E$ with a monotone normalized basis $(e_j)_{j=1}^k$ to be in the $k$-asymptotic structure can be equivalently described by having a winning strategy in the following game between to players: We fix $\vp>0$. Player I (the ``space chooser'') choses a space $X_1\in \cof(X)$, then player II (the ``vector chooser'') chooses a vector $x_1\kin S_{X_1}$, then player I and player II repeat  these moves to obtain spaces $X_1$, $X_2,\ldots, X_k$ in $\cof(X)$ and vectors $x_1,x_2, \ldots,x_k$, with $x_i\in S_{X_i}$. The space $E$ being in $\{X\}_k$  means that for every $\vp>0$ player II has a winning strategy, if his or her goal is to obtain a sequence $(x_j)_{j=1}^k$ which is $(1+\vp)$-equivalent to $ (e_j)_{j=1}^k$.

   For $E\kin\cE_k$ with monotone basis $(e_j)_{j=1}^k$ and $\vp\kgr0$ a winning strategy for the vector chooser can then be defined to be a {\em tree family}
   $$\cF=\big(x(X_1, X_2,\ldots, X_l): 1\le l\le k,\, X_1,X_2,\ldots, X_l\in \cof(X)\big)\subset S_X$$
    with the property that for any choice of  $X_1,X_2,\ldots, X_k\in \cof(X)$,  and any $l\le k$, $x(X_1,X_2,\ldots, X_l)\in S_{X_l}$ and
    so that the sequence $\big(x(X_1,X_2,\ldots, X_l)\big)_{l=1}^k$ is $(1+\vp)$-equivalent to $(e_j)_{j=1}^k$.

Since the game has finitely many steps it is {\em determined,} meaning that either the vector chooser or the space chooser has a winning strategy.
     Using the language of the game and its determinacy it is then easy to see that the set  $\{X\}_k$  is  the smallest compact subset for which the space chooser has a winning strategy if  for a given $\vp>0$
     his or her goal is that the  resulting sequence $(x_j)_{j=1}^k$ is at distance at most $\vp$ to $\{X\}_k$ (with respect to the metric $\delta_k$). In particular a Banach space is asymptotically $\ell_p$, $1\le p<\infty$,
     or asymptotically $c_0$,  if and only if there is a $c>0$ so that for each $k\in\N$ the space chooser has a winning strategy to  get a sequence $(x_j)_{j=1}^k$ which is $c$-equivalent to the unit vector basis in
     $\ell^k_p$, or $\ell^k_\infty$, respectively.

\item[d)]
 Assume that $X$ is a space with a separable dual. Then we can replace in the definition of  $\{X\}_k$ the  set $\cof(X)$ by a countable subset of $\cof(X)$, namely by the
 set
 $$ \big\{ F^\perp: F\subset \{{x^*_j:j\in\N}\} \text{ finite}\big\}, \text{ where $\{x^*_j:j\in\N\}\subset S_{X^*}$ is dense.}$$
In that case normalized  weakly null trees in $X$ indexed by $[\N]^{\le k}$ can be used to describe the $k$-th asymptotic structure: If $X^*$ is separable and $k\kin\N$, a space $E\in\cE_k$ with monotone basis $(e_j)_{j=1}^k$ is in $\{X\}_k$ if and only if for every $\vp>0$
   there is an unrooted  weakly null tree $\cT=\big(x_{\nb}:\nb\in[\N]^{\le k}\setminus\{\emptyset\}\big)$ in $S_X$ for which all branches are $(1+\vp)$-equivalent to $(e_j)_{j=1}^k$.

   It follows therefore from (c) and Ramsey's Theorem that $X$ is asymptotically $\ell_p$, for $1\le p<\infty$, or asymptotically $c_0$, if  there is a $C\ge 1$ so that for every $k\kin\N$ every unrooted  normalized weakly null tree of height $k$  has
   a refinement (as introduced in subsection \ref{SS:2.1}) all of whose branches are $C$-equivalent to  the $\ell_p^k$-unit vector basis.
 \end{enumerate}
\end{remarks}

The following observation will reduce the proof of the main results to the separable case.
\begin{prop}\label{P:3.3}
Let $X$ be a reflexive Banach space. Then there exists a separable subspace $Y$ of $X$ so that for all $k\in\N$ we have $\{X\}_k = \{Y\}_k$.
\end{prop}

We will need the following two lemmas first.

\begin{lem}\label{L:8.1}
Let $X$ be an infinite dimensional Banach space and let $E$ be a $k$-dimensional Banach space with a normalized monotone Schauder basis $(e_i)_{i=1}^k$. If for every $\vp>0$ there exists a weakly null tree $\{x_{\nb}:\nb\in[\N]^{\leq k}\setminus \{\emptyset\}\}\subset S_X$ so that for every $\mb=\{m_1,\ldots,m_k\}\in [\N]^k$ the sequence $(x_{\{m_1,\ldots,m_i\}})_{i=1}^k$ is $(1+\vp)$-equivalent to $(e_i)_{i=1}^k$. Then $(e_i)_{i=1}^k$ is in $\{X\}_k$.
\end{lem}

\begin{proof}
Recall that if $Y\in\mathrm{cof}(X)$ and $(z_i)_{i=1}^\infty$ is a normalized weakly null sequence then $\lim_i\mathrm{dist}(z_i,S_Y) \keq 0$. Fixing $\vp\kgr0$ and $k\in\mathbb{N}$ we will show that the vector player can choose a sequence that is $(1+\varepsilon)$-equivalent to $(e_i)_{i=1}^k$. Take a weakly null tree $(x_{\bar m}:\bar m\in[\N]^{\leq k})\subset S_X$  so that for all $\bar m = \{m_1,\ldots,m_k\}$ the sequence is $(x_{\{m_1,\ldots,m_i\}})_{i=1}^k$ is $(1+\delta)$-equivalent to $(e_i)_{i=1}^k$, where we will choose $\delta>0$ later. For each turn $1\leq i\leq k$ of the game when the subspace player chooses $Y_i\in\mathrm{cof}(X)$ the vector player picks $m_i>m_{i-1}$ (where $m_0 = 0$) so that there is $x_i\in S_{Y_i}$ with $\|x_i-x_{\{m_1,\ldots,m_i\}}\| \leq \delta$. For $\delta$ sufficiently small, this strategy for choosing $x_i$ in $S_{Y_i}$ insures that the sequence $(x_i)_{i=1}^k$ is $(1+\vp)$-equivalent to $(e_i)_{i=1}^k$.
\end{proof}

\begin{lem}
\label{L:8.2}
Let $X$ be a reflexive Banach space, $k\in\N$, $(e_i)_{i=1}^k\in\{X\}_k$, and let $\vp>0$. Then there exists a countably branching weakly null tree $\{x_{\nb}:\nb\in[\N]^{\leq k}\setminus \{\emptyset\}\}$ in $S_X$, all of whose branches are $(1+\vp)$-equivalent to $(e_i)_{i=1}^k$.
\end{lem}

\begin{proof}
We recall that the Eberlein-\v{S}mulyan theorem insures that if $W$ is a relatively weakly compact set in a Banach space and $x_0\in\overline{W}^w$ then there exists a sequence $(x_i)_{i=1}^\infty$ in $W$ with $x_i\stackrel{w}{\rightarrow}x_0$. Let $\vp>0$ and let  $\big( x(Y_1,Y_2,\ldots, Y_i): i=1,2,\ldots, k, Y_1,Y_2,\ldots, Y_i\in \cof(X)\big)$ be a normalized tree with $x(Y_1,Y_2,\ldots, Y_i)\in S_{Y_i}$ and whose branches approximate $(e_j)_{j=1}^k$ up to $(1+\vp)$-equivalence (see Remarks \ref{R:3.2} (c)). By reflexivity the set $\{x(Y): Y\in\mathrm{cof}(X)\}$ (the first level of  the tree) is relatively weakly compact. Also, $0\in\overline{\{x(Y): Y\in\mathrm{cof}(X)\}}^w$. Indeed, if $f_1,\ldots,f_d$ are in $X^*$ then $Y = \bigcap_{j=1}^d \ker(f_j)$ is in $\mathrm{cof}(X)$ and hence $f_j(x(Y)) = 0$ for $1\leq j\leq d$. We may thus pick a sequence $(Y_l)_l$ in $\mathrm{cof}(X)$ with $x(Y_l)\stackrel{w}{\rightarrow} 0$.

Assume that for some $i\in\N$ we have assigned for each $\{m_1,\ldots,m_i\}\in[\N]^i$ a vector $x_{\{m_1,\ldots,m_i\}}$ of the form $x(Y_{m_1},Y_{m_2},\ldots,Y_{m_i})$.  As before, we may pick a sequence $(Y^{(i+1)}_l)_l$ so that $x(Y_{m_1},Y_{m_2},\ldots,Y_{m_i},Y^{(i+1)}_l)\stackrel{w}{\rightarrow}0$. For $j>m_i$ we define $$x_{\{m_1,\ldots,m_i,j\}} = x(Y_{m_1},Y_{m_2},\ldots,Y_{m_i},Y^{(i+1)}_l),$$
for some large enough $l$. Thus, every $(x_{\{m_1,\ldots,m_i\}})_{i=1}^k$ is of the form $\big(x(Y_1,\ldots,Y_i)\big)_{i=1}^k$, and thus $(1+\vp)$-equivalent to $(e_j)_{j=1}^k$.
\end{proof}

\begin{proof}[Proof of Proposition \ref{P:3.3}]
Since for every $k\in\N$ the $k$-asymptotic structure $\{X\}_k$ is separable (with respect to the metric introduced in Section \ref{S:3} (b)), we can find a countable set $\{ (e^{(l)}_j)_{j=1}^k:l\in\N\}\subset \{X\}_k$ which is dense in $\{X\}_k$ and, using Lemma \ref{L:8.2}, a countable collection of weakly null trees $\big\{( x^{(r)}_{\nb}:\nb\in[\N]^{\le k}): r\in\N\big\}$ in $S_X$ so that for  each $\vp>0$
and each  $l\in\N$ there is a $r\in\N$,  so that for all $\nb\in[\N]^k$,  the sequence $\big( x^{(r)}_{\mb}: \mb\preceq\nb) $ is $(1+\vp)$-equivalent to $(e^{(l)}_j)_{j=1}^k$. We define  $Y_k$ to be the closed linear span of  $\{x^{(r)}_{\nb}:r\in\N, \nb\in[\N]^{\le k}\}$. Since $\{Y_k\}_k$ and $\{X\}_k$ are  compact (see (b) in Section \ref{S:3}) it follows that $\{Y_k\}_k=\{X\}_k$. Finally we conclude our proof by setting $Y$ to be the closed linear span of $\cup_{k\in \N}Y_k$ and deduce our claim.
\end{proof}

We now turn to ``sequential asymptotic properties'' of Banach spaces. These are properties which involve  sequences  and their subsequences, as opposed to trees and their refinements.

Let $X$ be a Banach space and $k\in\N$. A family $\big(x^{(i)}_j:i=1,2,\ldots, k, j\in \N\big)\subset X,$ is called an {\em array of height $k$ in $X$}. An {\em array of infinite height} in $X$ is a family $\big(x^{(i)}_j:i, j\in \N\big)\subset X$.

For (finite or infinite) arrays $\big(x^{(i)}_j:i=1,2,\ldots, k, j\in \N\big)$, or $\big(x^{(i)}_j:i, j\in \N\big)$ respectively, we call the sequence  $(x^{(i)}_j)_{j\in\N}$  {\em the $i$-th row of the array}. We call an array weakly null  if all rows are weakly null. A {\em subarray} of a finite array $\big(x^{(i)}_j:i=1,2,\ldots, k, j\in \N\big)\subset X,$ or an infinite array $\big(x^{(i)}_j:i\in \N, j\in \N\big)\subset X,$ is an array of the form $\big(x^{(i)}_{j_s}:i=1,2,\ldots, k, s\in \N\big)$ or  $\big(x^{(i)}_{j_s}:i\in \N, s\in \N\big)$, respectively,
where $(j_s)\subset \N$ is a subsequence.  Thus, for a subarray we
are taking  the same subsequence in each row.

 The following  notion was introduced by Halbeisen and Odell \cite{HalbeisenOdell2004}.
\begin{defin}\label{D:3.4}\cite{HalbeisenOdell2004} A basic sequence $(e_i)$  is called an {\em asymptotic model} of a Banach space $X$, if there exist
 an    infinite  array  $\big(x^{(i)}_j:i ,j\kin \N\big)\subset S_X$ and a   null-sequence $(\vp_n)\subset(0,1)$, so that for  all $n$, all $(a_i)_{i=1}^n \subset[-1,1]$ and $n\le k_1<k_2<\ldots  <k_n$, it follows that
 \begin{equation*}
 \Bigg|\Big\|\sum_{i=1}^n a_i x^{(i)}_{k_i}\Big\|  - \Big\| \sum_{i=1}^n a_i e_i \Big\|\Bigg|<\vp_n.
 \end{equation*}
  \end{defin}
  In \cite{HalbeisenOdell2004} the following was shown.
\begin{prop}\label{P:3.5}\cite{HalbeisenOdell2004}*{Proposition 4.1 and Remark 4.7.5}
Assume that  $\big(x^{(i)}_j:i, j\in \N\big)\subset S_X$ is an infinite array, all of whose rows  are normalized and weakly null. Then there is  a subarray of $\big(x^{(i)}_j:i, j\in \N\big)$ which has a $1$-suppression unconditional  asymptotic model $(e_i)$.
\end{prop}

  We call a basic sequence $(e_i)$
{\em $c$-suppression unconditional,} for some $c\ge 1$, if
for any $(a_i)\subset c_{00}$ and any $A\subset \N$
$$\Big\|\sum_{i\in A} a_i e_i\Big\|\le c \Big\|\sum_{i=1}^\infty a_i e_i\Big\|.$$
We call $(e_i)$ $c$-unconditional if for any $(a_i)\subset c_{00}$ and any $(\sigma_i)\in\{\pm 1\}^\N$
$$\Big\|\sum_{i=1}^\infty   a_i e_i\Big\|\le c \Big\|\sum_{i=1}^\infty  \sigma_i a_i e_i\Big\|.$$
Note that a $c$-unconditional basic sequence is $c$-suppression unconditional.

The following important result  was shown in \cite{FOSZ2017} and it is an integral ingredient of the proof of Theorem \ref{thmB}.
\begin{thm}\label{T:3.6} \cite{FOSZ2017}*{Theorem 4.6}
If a separable Banach space $X$ does not contain any isomorphic copy of $\ell_1$ and    all the asymptotic models generated by normalized weakly null arrays are
equivalent to the $c_0$ unit vector basis, then $X$ is asymptotically $c_0$.
\end{thm}

Asymptotic models can be seen as a  generalization of {\em spreading models}, a notion which was introduced much earlier by Brunel  and Sucheston  \cite{BrunelSucheston1974}.
Spreading models are asymptotic models for arrays with identical rows.

\begin{defin}\label{D:3.7}  \cite{BrunelSucheston1974}
Let $E$ be a Banach space with a normalized basis $(e_i)$ and let
$(x_i)$ be a basic sequence in a Banach space $X$.   We say that $E$ with its basis $(e_i)$ is a {\em spreading model of $(x_i)$},  if  there is a   null-sequence $(\vp_n)\subset(0,1)$, so that for  all $n$, all $(a_i)_{i=1}^n \subset[-1,1]$ and $n\le k_1<k_2<\ldots<k_n$, it follows that
 \begin{equation*}
 \Bigg|\Big\|\sum_{i=1}^n a_i x_{k_i}\Big\|_X  - \Big\| \sum_{i=1}^n a_i e_i \Big\|_E\Bigg|<\vp_n
 \end{equation*}
or, in other words, if
$$\lim_{k_1\to\infty}\lim_{k_2\to\infty} \ldots \lim_{k_n\to\infty} \Big\|\sum_{j=1}^n a_j x_{k_j}\Big\|_X=\Big\|\sum_{j=1}^n a_j e_j\Big\|_E.$$
\end{defin}
Using Ramsey's Theorem it is easy to see that every normalized basic sequence has a subsequence which admits a spreading model, which of course also follows form the above cited result in \cite{HalbeisenOdell2004}. A spreading model $E$ with basis $(e_i)$   generated by a normalized weakly null sequence is {\em 1-suppression unconditional} \cite{BeauzamyLapreste1984}*{Proposition 1, p. 24}.

Let $k\in\N$ and let $\big(x^{(i)}_j:i=1,2,\ldots, k, j\in \N\big)\subset S_X$   be a normalized  weakly null array of height $k$. We extend this array to an infinite array $\big(x^{(i)}_j:i\in \N, j\in \N\big)$, by letting $$x^{(sk+i)}_j=x^{(i)}_j, \text{ for $s\in\N$ and $i=1,2, \ldots, k$.}$$ By Proposition \ref{P:3.5}  we can pass to a subarray  $(z^{(i)}_{j}: i\in\N, j\in\N)$  of  $(x^{(i)}_{j}: i\in\N, j\in\N)$ which admits an asymptotic model $(e_j)$.
Now letting $e^{(i)}_{j}= e_{(j-1)k+ i}$, for $i=1,2, \ldots, k$ and $j\in\N$ we observe that
the array $(e^{(i)}_j)_{i,j\in\N}$ is the {\em joint spreading model of  $(z^{(i)}_{j}: i\in\N, j\in\N)$}, a notion introduced and discussed  in \cite{AGLM2017}. We recall the definition of {\em joint spreading models} and will first recall the definition of {\em plegmas}.

\begin{defin}\label{D:3.8} \cite{ArgyrosKanellopoulosTyros2013}*{Definition 3}
Let $k,m\in\N$ and $s_i = (s^{(i)}_1, s^{(i)}_2,\ldots,s^{(i)}_m)\subset \N$ for $i=1,\ldots, k$. The family $(s_i)_{i=1}^k$ is called a {\em plegma } if
$$s^{(1)}_1 < s_1^{(2)} < \cdots< s_1^{(k)} < s_2^{(1)}< s_2^{(2)}<\cdots< s_2^{(k)}<\cdots<s_m^{(1)} < s^{(2)}_m < \cdots < s^{(k)}_m.$$
\end{defin}

\begin{defin}\label{D:3.9} \cite{AGLM2017}*{Definition 3.1}
Let   $\big(x_j^{(i)}: 1\kleq i\kleq k, j\kin\N\big) $ and $\big(e_j^{(i)}: 1\kleq i\kleq  k, j\kin\N\big)$ be two normalized arrays in the Banach spaces $X$, and $E$, respectively, whose rows are normalized and basic. We say that  $(x_j^{(i)}:1\kleq i\kleq  k,  j\kin\N) $  {\em generates  $(e_j^{(i)}: 1\kleq i\kleq  k, j\kin\N)$  as a joint spreading model} if there exists a null sequence of positive real numbers $(\vp_m)_{m=1}^\infty$ so that for every $m\in\N$, every  plegma $(s_i )_{i=1}^k$, $s_i=(s^{(i)}_j :j=1,2,\ldots, m)$ for $1\leq i\leq k$, with $\min(s_1)=s^{(1)}_1 \geq m$, and scalars $((a_j^{(i)})_{j=1}^m)_{i=1}^k$ in $[-1,1]$ we have
\[\Bigg|\Big\|\sum_{j=1}^m\sum_{i=1}^k a_j^{(i)}x^{(i)}_{s^{(i)}_j}\Big\|_X - \Big\|\sum_{j=1}^m\sum_{i=1}^ka_j^{(i)}e^{(i)}_{j}\Big\|_E\Bigg|<\vp_m.\]
\end{defin}

\begin{rem}\label{Rem:3.10}
Note that if  $\big(x_j^{(i)}: 1\kleq i\kleq  k, j\kin\N\big) $ generates  $\big(e_j^{(i)}:  1\kleq i\kleq  k, j\kin\N\big)$ as a joint spreading model, then  $(e_j^{(i)})_{j=1}^\infty$   is a spreading model of $(x^{(i)}_j)_{j=1}^\infty$, for $i=1,2,\ldots, k$.
\end{rem}

In the next remark we discuss the differences between asymptotic and sequential asymptotic properties.
\begin{remark}\label{Rem:3.11}  Assume that  $X$ is a separable reflexive space. Then, by observation (d) in Remarks \ref{R:3.2}, the property that $X$ is asymptotically $\ell_p$, for some $1\le p\le \infty$ (as usual replace $\ell_\infty$ by $c_0$ if $p=\infty$), is equivalent to the property that there is a $C\ge 1$, so that for every $k\in \N$ every weakly null tree $(x_{\nb}:\nb\in[\N]^{\le k})$  of height $k$  can be refined (as defined in subsection \ref{SS:2.1}) to a tree $(x_{\mb}:\mb\in[\M]^{\le k })$, $\M\in[\N]^\omega$, which has the property that each branch is $C$-equivalent to the $\ell_p^k$ unit vector basis.

Secondly we consider  the property of a Banach space $X$ that every asymptotic model generated by a weakly null array is $C$-equivalent to  the $\ell_p$-unit vector basis, for some $1\le p<\infty$,  or  the $c_0$-unit vector basis. For a normalized  weakly null array  $(x^{(i)}_j: i,j\in\N)$ we put $x_{\mb}=x^{(i)}_{\max(\mb)}$ for $\mb\in [\N]^{i}$ and call for $k\in\N$ the tree $(x_{\nb}:\nb\in[\N]^{\le k})$ {\em the tree of height $k$ generated by the array $(x^{(i)}_j: i,j\in\N)$}. Note that $x_{\nb}$ for $\nb\in[\N]^{\le k}$ only depends on $\max(\nb)$ and the cardinality of $\nb$, but not on the predecessors of $\nb$. Then, by a straightforward diagonalization argument, one shows that the property that every asymptotic model generated by a weakly null array is $C$-equivalent to  the $\ell_p$-unit vector basis for some $C \ge 1$, is equivalent with the property that every tree of height $k$, generated by a normalized weakly null array  has a refinement all of whose branches are $C$-equivalent to the $\ell_p^k$-unit vector basis, for some $C \ge 1$.

Thus, the property that the asymptotic models generated by normalized weakly null arrays is $C$-equivalent to  the $\ell_p$-unit vector basis, is a property of specific weakly null trees. Theorem \ref{T:3.6} is therefore  a surprising result, and its proof relies on the fact that the $c_0$-norm is somewhat extremal. Usually it is not possible to deduce from  a sequentially asymptotic property of a Banach space an asymptotic property. For example, in a forthcoming paper, we build a reflexive space $X$, all of whose asymptotic models are isometrically equivalent to the $\ell_2$-unit vector basis, but for given $1\le p\le \infty$, $p\not=2$, $X$ has $\ell_p^n$ in its $n$-th asymptotic structure.



\end{remark}

\section{Proof of Theorems A and B}\label{S:4}
This section is devoted to proving Theorem \ref{thmB} and then obtaining Theorem \ref{thmA} as a corollary.
The proof is based on the main argument of \cite{BaudierLancienSchlumprecht2018} and on the above cited result in \cite{FOSZ2017} (see Thoerem \ref{T:3.6} in our paper) that  connects asymptotic properties with properties of arrays.

The following  lemma  includes  a  well known refinement  argument which is crucial for the proof of the main result. For completeness we include a proof.

\begin{lem}\label{L:4.1}  Let $X$ be a reflexive Banach space,  $k\in\N$,  and $f: [\N]^k\to X$  have a bounded image.
Then there exist $\M\kin[\N]^\omega$ and a weakly null tree $(y_{\mb}:\mb\in[\M]^{\le k})$,
  so that
  $f(\mb)=y_{\emptyset}+\sum_{i=1}^k y_{\{m_1,\ldots, m_i\}}$,  for all $\mb\in [\M]^k$.

  Moreover, if we equip $[\N]^k$ with $d^{(k)}_\H$ then for all $\mb\in [\M]^{\leq k}\setminus\{\emptyset\}$ we have $\|y_{\mb}\| \leq \mathrm{Lip}(f)$.
\end{lem}

\begin{proof} We prove the claim by induction for all $k\kin\N$.  If $k=1$, we can take a subsequence   $(x_{n})$ of $\big(f(\{n\})\big)_{n\in\N}$ which converges to some $y_\emptyset\kin X$. Then put $y_{\{n\}}=x_n-y_\emptyset$.

Assume our claim to be true for $k-1$, with  $k\kin\N$, and let $f: [\N]^{k}\to X$ have a bounded image.  We put $l_i=i$, for $i=1,2,\ldots, k-1$, and choose $\L_{k-1}\in [\{k,k+1,\ldots \}]^\omega$ so that $x_{\{1,2,\ldots, k-1\}}=w-\lim_{l\to\infty, l\in \L_{k-1}} f(\{1,2,\ldots, k-1\}\cup\{l\})$ exists. Then we can recursively choose for each $n \geq k$, $l_n\in\N$, $\L_n\in[\L_{n-1}]^\omega$, with $l_n\in \L_{n-1}$ and $l_n<\min(\L_n)$, so that for each $\mb\subset \{l_1,l_2,\ldots, l_n\}$, with $\#\mb=k-1$, $x_{\mb}=w-\lim_{l\to\infty, l\in \L_n} f(\mb\cup\{l\})$ exists. Let $\L=\{l_j:j\kin\N\}$ and put $y_{\mb}=f(\mb)- x_{\{m_1,m_2,\ldots, m_{k-1}\}}$  for $\mb\keq\{m_1,m_2,\ldots, m_{k}\}\kin[\L]^{k} $.

Finally we apply the induction hypothesis to $f':[\L]^{k-1}\to X$, $\mb\mapsto x_{\mb}$, which provides us with an infinite $\M\subset \L$ and  a weakly null tree $(y_{\mb}:\mb\in[\M]^{k-1})$ so that $x_{\mb}=\sum_{i=0}^k y_{\{m_1,m_2,\ldots, m_i\}}$ for all ${\mb}=\{m_1,m_2,\ldots, m_{k-1}\}\in[\M]^{k-1}$ and thus,
\[f(\mb)=y_{\mb}+x_{\{m_1,m_2,\ldots, m_{k-1}\}}=\sum_{i=0}^{k} y_{\{m_1,m_2,\ldots, m_i\}}
\text{ for all $\mb=\{m_1,m_2,\ldots, m_{k}\}\in[\M]^{k}$.}\]

To prove the second part of the statement let $\mb = \{m_1,m_2,\ldots, m_i\}$ in $[\M]^k\setminus \{\emptyset\}$ and put $\mb'=\{m_1,m_2,\ldots, m_{i-1}\}$. It follows from the lower semicontinuity of the norm with respect to the weak topology that
\begin{align*}
\|y_{\mb}\|&=
\Big\|w\text{\;-}\lim_{n_i\to\infty}\lim_{n_{i+1}\to\infty}\!\!\!\!\!\cdots\!\!\! \lim_{n_k\to\infty}\Big( f(\mb\cup\{n_{i+1},\ldots, n_k\})- f(\mb'\cup\{n_i,n_{i+1},\ldots, n_k\})\Big)\Big\|\\
  &\le \limsup_{n_i\to\infty}\limsup_{n_{i+1}\to\infty}\cdots \limsup_{n_k\to\infty}\big\|f(\mb\cup\{n_{i+1},\ldots, n_k\})- f(\mb'\cup\{n_i,n_{i+1},\ldots, n_k\})\big\|\\
  &\le \limsup_{n_i\to\infty}\limsup_{n_{i+1}\to\infty}\cdots  \limsup_{n_k\to\infty} \mathrm{Lip}(f)d^{(k)}_\H(\mb\cup\{n_{i+1},\ldots, n_k\},\mb'\cup\{n_i,n_{i+1},\ldots, n_k\})\\
  &= \mathrm{Lip}(f).
\end{align*}
\end{proof}

For the proof of Theorem \ref{thmB}  a slightly weaker version of the next result would be sufficient, but its proof would not be significantly easier.

\begin{lem}\label{L:4.2}
Let $X$ be a $C$-asymptotic-$c_0$ Banach space for some $C\ge 1$, $k\in\N$, and let also $(x_{\nb}:\nb\in[\N]^{\leq k})$ be a bounded weakly null tree. Then for every $\vp>0$ there exists $\L\in[\N]^{\omega}$ so that for every $\mb=\{m_1,\ldots,m_k\}\in[\L]^k$ and every $F\subset\{1,\ldots,k\}$ we have \[\Big\|\sum_{i\in F}x_{\{m_1,\ldots,m_i\}}\Big\| \leq (C+1+\vp)\max_{i\in F}\|x_{\{m_1,\ldots,m_i\}}\|.\]
\end{lem}

\begin{proof}
We will just find one such $\mb$. This is sufficient by Ramsey's theorem, since such a set $\mb$ could be found in each infinite subset of $\N$. Let us play a $k$-round vector game in which the subspace player follows a winning strategy to force the vector player to choose a sequence $(C+\vp)$-equivalent to the unit vector basis of $\ell_\infty^k$. In each step $i$ the subspace player picks a subspace $Y_i$ of finite codimension according to his or her winning strategy. The vector player picks $y_i\in Y_i$ according to the following scheme: recursively pick $m_1<\cdots<m_k$ so that one of the following holds:
\begin{itemize}
\item[(a)] If $\limsup_n\|x_{\{m_1,\ldots,m_{i-1},n\}}\| > 0$ pick $m_i$ (with $m_i>m_{i-1}$ if $i>1$) and $y_i$ in the unit sphere of $Y_i$ so that
$$\Big\|y_i - \frac{x_{\{m_1,\ldots,m_i\}}}{\|x_{\{m_1,\ldots,m_i\}}\|}\Big\| < \vp2^{-i} .$$
In the above argument we have used the following corollary of the Hahn-Banach Theorem. If $Y\in\mathrm{cof}(X)$ and $(z_i)_{i=1}^\infty$ is a weakly null sequence then $\lim_i\mathrm{dist}(z_i,Y) \keq 0$. If in particular $(z_i)_i$ is normalized then $\lim_i\mathrm{dist}(z_i,S_Y) \keq 0$.

\item[(b)] If $\lim_{n}\|x_{\{m_1,\ldots,m_{i-1},n\}}\| = 0$ we distinguish between the following subcases:
\begin{itemize}
\item[(b1)]  if   $i=1$ or $x_{\{m_1,\ldots,m_j\}} = 0$, for  all $1\leq j<i$,  pick arbitrary $m_{i}$, so that  $m_i>m_{i-1}$ if $i>1$, and arbitrary $y_i$ in the unit sphere of $Y_i$, and
\item[(b2)] if $i>1$ and $x_{\{m_1,\ldots,m_j\}} \neq 0$, for some $1\leq j<i$, pick $m_i>m_{i-1}$ so that
$$\|x_{\{m_1,\ldots,m_i\}}\| < (\vp2^{-i})\min\big\{\|x_{\{m_1,\ldots,m_j\}}\|: 1\leq j< i\text{ with } x_{\{m_1,\ldots,m_j\}}\neq 0\big\}$$
and pick an arbitrary $y_i$ in the unit sphere of $Y_i$.
\end{itemize}
\end{itemize}
It follows that the sequence $(y_i)_{i=1}^k$ is $(C+\vp)$-equivalent to the unit vector basis of $\ell_\infty^k$.

Let now $F\subset\{1,2,\ldots,k\}$. Set
\begin{align*}
F_1 = \{i\in F:\text{ (a) is satisfied}\} \text{ and }
\bar F_2 = \{i\in F:\text{ (b) is satisfied and }x_{\{m_1,\ldots,m_i\}}\neq 0\}.
\end{align*}
Set $i_0 = \min(\bar F_2)$ and $F_2 = \bar F_2\setminus\{i_0\}$, if $\bar F_2$ is non-empty, otherwise let $F_2=\emptyset$. We now calculate
\begin{align*}
\Big\|\sum_{i\in F}x_{\{m_1,\ldots,m_i\}}\Big\| &\leq \Big\|\sum_{i\in F_1}x_{\{m_1,\ldots,m_i\}}\Big\| + \|x_{\{m_1,\ldots,m_{i_0}\}}\| +  \Big\|\sum_{i\in F_2}x_{\{m_1,\ldots,m_i\}}\Big\|\\
&\leq \Big\|\sum_{i\in F_1}\|x_{\{m_1,\ldots,m_i\}}\|y_i\Big\| + \Big\|\sum_{i\in F_1}x_{\{m_1,\ldots,m_i\}} - \|x_{\{m_1,\ldots,m_i\}}\|y_i\Big\|\\
&\phantom{=} + \|x_{\{m_1,\ldots,m_{i_0}\}}\| + \sum_{i\in F_2}\frac{\vp}{2^i} \|x_{\{m_1,\ldots,m_{i_0}\}}\|\\
&\leq (C\kplus\vp)\max_{i\in F_1}\|x_{\{m_1,\ldots,m_i\}}\| + \!\sum_{i\in F_1}\frac{\vp}{2^i}\|x_{\{m_1,\ldots,m_i\}}\|
+ (1\kplus\vp)\|x_{\{m_1,\ldots,m_{i_0}\}}\|\\ &\leq (C+1+3\vp)\max_{i\in F}\|x_{\{m_1,\ldots,m_i\}}\|.
\end{align*}
An adjustment of $\vp$ yields the desired estimate.
\end{proof}

The following is one of the main statements presented in this paper.

\begin{thm}[Theorem \ref{thmB}]
\label{T:4.3}
A Banach space $X$ is reflexive and asymptotic-$c_0$ if and only if there exists $C\geq 1$ satisfying the following: for every $k\in\N$ and Lipschitz map $f:([\N]^k,d^{(k)}_\H)\to X$ there exists $\L\in[\N]^\omega$ so that
\begin{equation}
\label{E:4.3.1}
\diam\big( f([\L]^k)\big) \leq C\mathrm{Lip}(f).
\end{equation}
\end{thm}

\begin{proof}
We first assume that $X$ is reflexive and $B$-asymptotically $c_0$. Let $k\kin\N$ and let $f:([\N]^k,d^{(k)}_\H)\to X$ be a Lipschitz map. By Lemma \ref{L:4.1} there exist $\M\kin[\N]^\omega$ and a weakly null tree $(y_{\mb}:\mb\in[\M]^{\le k})$ so that  $f(\mb)=\sum_{\lb\preceq\mb} y_{\lb}$, for all $\mb \kin[\M]^{k}$, and   $\|y_{\mb}\| \leq \mathrm{Lip}(f)$, for all $\mb \kin[\M]^{\leq k}\setminus \{\emptyset\}$.
By Lemma \ref{L:4.2} we find $\L\in[\M]^\omega$, so that
$$\Big\|\sum_{i\in F}y_{\{m_1,\ldots,m_i\}}\Big\| \leq (B+2)\max_{i\in F}\|y_{\{m_1,\ldots,m_i\}}\| .$$
for all $\mb=\{m_1,m_2,\ldots, m_k\}\in[\L]^k$ and $F\subset\{1,\ldots, k\}$.Thus, for $\mb,\nb$ in $[\L]^k$ we have
\[\|f(\mb) - f(\nb)\| = \Big\|\sum_{\ub\preceq \mb}y_{\ub} - \sum_{\vb\preceq\nb}y_{\vb}\Big\|\leq \Big\|\sum_{\emptyset\prec\ub\preceq \mb}y_{\ub} \Big\|+\Big\|\sum_{\emptyset\prec\vb\preceq\nb}y_{\vb}\Big\|\leq 2(B+2)\mathrm{Lip}(f)\]
and so for $C = 2(B+2)$ the conclusion is satisfied.

To prove the converse, we  show that if either $X$ is not reflexive or $X$ is reflexive and not asymptotic-$c_0$, then there exists a  sequence $(f_k)$, $f_k:([\N]^k,d^{(k)}_\H)\to X$, $\text{Lip}(f_k)\le 1$, for $k\in\N$,
and
\begin{equation}\label{E:4.3.2}\inf_{\M\in[\N]^\omega} \diam\big(f_k([\M]^k)\big)\nearrow \infty, \text{ if $k\nearrow\infty$}.\end{equation}

Assume first that $X$ is non-reflexive. By James' characterization of reflexive spaces \cite{James1972}, there exists a sequence $(x_n)\subset B_X$ such that for all $k\ge 1$ and $\mb=\{m_1,m_2,\ldots, m_{2k}\}\in[\N]^{2k}$,
\begin{equation}\label{E:4.3.3}
\Big\|\sum_{i=1}^kx_{m_i}-\sum_{i=k+1}^{2k}x_{m_i}\Big\|\ge \frac{k}{2}.
\end{equation}
Define $f_k(\mb)=\frac12\sum_{i=1}^kx_{m_i}$, for $\mb=\{m_1,\ldots, m_k\}$ in $[\N]^k$. This map is $1$-Lipschitz with respect to $d^{(k)}_\H$
and \eqref{E:4.3.3} implies \eqref{E:4.3.2}.

Secondly assume  $X$ is  reflexive and not asymptotically-$c_0$.  By Proposition \ref{P:3.3} there is a separable subspace  of $X$ that is not asymptoptically-$c_0$, so we can assume that $X$ is separable.
By Theorem \ref{T:3.6} there exists a 1-suppression unconditional sequence $(e_i)_i$ that is not equivalent to the unit vector basis of $c_0$, and hence  $\lambda_k=\|\sum_{i=1}^ke_i\| \nearrow\infty$, if $k\nearrow\infty$, and that is generated as an asymptotic model of a normalized weakly null array $(x^{(i)}_j: i,j\kin\N)$ in $X$. Fixing $k \kin \N$ and $\delta >0$ and after passing to appropriate subsequences of the array, we may assume that for any  $k\le j_1<\cdots<j_k$ and any $a_1,\ldots,a_k$ in $[-1,1]$ we have
\begin{equation}\label{E:4.3.4}
\Bigg|\Big\|\sum_{i=1}^k a_ix_{j_i}^{(i)} \Big\|-\Big\|\sum_{i=1}^k a_i e_i \Big\|\Bigg|<\delta.
\end{equation}
Define now $f_k(\mb)=\frac12\sum_{i=1}^k x_{m_i}^{(i)}$ for $\mb=\{m_1,\ldots,m_k\}\in [\N]^k$. Note that $f$ is 1-Lipschitz for the metric $d_{\H}^{(k)}$.\\
Then, if $\mb=\{m_1,\ldots,m_k\}$, $\nb=\{n_1,\ldots,n_k\}$ and $F=\{i: m_i\neq n_i\}$ we have
$$f_k(\mb)-f_k(\nb)=\frac12\sum_{i\in F}x_{m_i}^{(i)} - \frac12\sum_{i\in F}x_{n_i}^{(i)}.$$
Using the fact that the array is weakly null and the Hahn-Banach theorem, for all $\M \in [\N]^\omega$, all $\mb$ in $[\M]^k$, we can find  $x^*\in S_{X^*}$ and $\nb \in [\M]^k$ such that
$$x^*\Big(\sum_{i\in F}x_{m_i}^{(i)} - \sum_{i\in F}x_{n_i}^{(i)}\Big)\ge
\Big\|\sum_{i\in F}x_{m_i}^{(i)}\Big\|-\delta.$$
Using equation (\ref{E:4.3.4}), we deduce that
$$\|f_k(\mb)-f_k(\nb)\|\ge \frac12\lambda_k-\delta.$$
If $\delta$ was chosen small enough, we obtain that for all $\M \in [\N]^\omega$  $\diam\big(f_k([\M]^k)\big)\ge \frac{\lambda_{k}}4$, which proves our claim.

\end{proof}

\begin{cor}[Theorem \ref{thmA}]
\label{C:4.4} Let $Y$ be a reflexive asymptotic-$\co$ Banach space. If $X$ is a Banach space that coarsely embeds into $Y$, then $X$ is also reflexive and asymptotic-$\co$.
\end{cor}

\begin{proof}
Let $g:X\to Y$ be a coarse embedding with moduli
$\rho_g,\omega_g:[0,\infty)\to[0,\infty)$.
By Theorem \ref{T:4.3} the space $Y$ satisfies \eqref{E:4.3.1}, for
some constant $C\geq 1$. It is enough to show that the same is true for
$X$ and some   $D\geq 1$ such that $\rho_g(D) > C\omega_g(1)$.

Let $f:[\N]^k\to X$ be a non-constant Lipschitz map.
Take $h:[\N]^k\to Y$ with $h(\mb) = g(\mathrm{Lip}(f)^{-1}f(\mb))$.
Because $d^{(k)}_\H$ is an unweighted graph metric it follows that
$$\mathrm{Lip}(h) = \omega_{h}(1) \leq \omega_g\big(\mathrm{Lip}(f)^{-1}\omega_f(1)\big) = \omega_g\big(\mathrm{Lip}(f)^{-1}\mathrm{Lip}(f)\big) = \omega_g(1).$$
Pick $\L\kin[\N]^\omega$ so that for all $\mb,\nb\kin[\L]^k$ we have $\|h(\mb) - h(\nb)\| \kleq C\omega_g(1)$. On the other hand,
\[
\begin{split}
C\omega_g(1) &\geq \|h(\mb) - h(\nb)\| = \big\|g\big(\mathrm{Lip}(f)^{-1}f(\mb)\big) - g\big(\mathrm{Lip}(f)^{-1}f(\nb)\big)\big\|\\
&\geq \rho_g\big(\mathrm{Lip}(f)^{-1}\|f(\mb) - f(\nb)\|\big).
\end{split}
\]
Thus, $\mathrm{Lip}(f)^{-1}\|f(\mb) - f(\nb)\|\leq D$, or $\|f(\mb) - f(\nb)\| \leq D \mathrm{Lip}(f)$, for any $\mb,\nb\in[\L]^k$.
\end{proof}

A simple re-scaling argument (see the end of section 4 in  \cite{BaudierLancienSchlumprecht2018}) allows us to adapt the above proofs in order to show the following.

\begin{cor}\label{C:4.5} Let $Y$ be a reflexive asymptotic-$\co$ Banach space. If $X$ is a Banach space such that $B_X$ uniformly embeds into $Y$, then $X$ is also reflexive and asymptotic-$\co$.
\end{cor}

\begin{remark}\label{R:7.5} For $k\in \N$, the Johnson graph of height $k$ is the set $[\N]^k$ equipped with the  metric defined by $d_{\J}^{(k)}(\mb,\nb)=\frac12 \sharp(\mb \Delta \nb)$ for $\mb,\nb \in [\N]^k$. It is proved in \cite{BaudierLancienSchlumprecht2018} that there is a constant $C\ge 1$ such that for any $k\in \N$ and $f \colon ([\N]^k,d_{\J}^{(k)}) \to T^*$ Lipschitz, there exists $\M \in [\N]^\omega$ so that $\diam (f([\M]^k))\le C\Lip (f)$. It is easily seen that the same is true if $T^*$ is replaced by any reflexive asymptotic-$c_0$ space. It is also clear that this concentration property for Lipschitz maps from the Johnson graphs implies the reflexivity of the target space. We do not know if it implies that it is asymptotic-$c_0$. We do not know either whether the equi-coarse embedabbility of the Johnson graphs and of the Hamming graphs are equivalent conditions for a Banach space. The reason is that canonical embeddings of the Johnson graphs are built on sequences and not arrays. This confirms the qualitative difference between asymptotic models and spreading models.
\end{remark}

\section{Quasi-reflexive asymptotic-$c_0$ spaces}\label{quasi-reflexive}
Let us first recall that a Banach space is said to be {\it quasi-reflexive}
if the image of its canonical embedding into its bidual is of finite
codimension in this bidual. For an infinite subset $\M$ of $\N$, we denote $I_k(\M)$ the set of strictly interlaced pairs in $[\M]^k$, namely :
$$I_k(\M)=\big\{(\mb,\nb) \in [\M]^k \times [\M]^k,\ m_1 < n_1 < m_2 < n_2 < \cdots < m_k < n_k\big\}.$$
Note that for $(\mb,\nb) \in I_k(\M)$, $d^{(k)}_\H(\mb,\nb)=k$.
Our next result mixes arguments from Lemma \ref{L:4.2} of this paper and of Theorem 2.2
in \cite{LancienRaja2018}.

\begin{thm} Let $C\ge 1$ and $X$ be a quasi-reflexive
$C$-asymptotic-$c_0$ Banach space. Then, for any Lipschitz map
$f:([\N]^k,d^{(k)}_\H) \to X$, there exists $\M \in [\N]^\omega$ such that
$$\forall (\mb,\nb)\in I_k(\M),\ \|f(\mb)-f(\nb)\|\le 3(C+1) \Lip(f).$$
In particular, the family $([\N]^k,d^{(k)}_\H)_{k\in \N}$ does not
equi-coarsely embed into $X$.
\end{thm}

\begin{proof} Let us write $X^{**}=X \oplus E$, where $E$ is a finite
dimensional space. Let $f:([\N]^k,d^{(k)}_\H) \to X$ be a Lipschitz map.
Since $f$ is countably valued and $X$ is quasi-reflexive, we may as well
assume that $X$ and therefore all its iterated duals are separable.
We may also assume that $\Lip(f)>0$.
Then mimicking the proof of Lemma \ref{L:4.1} and using weak$^*$-compactness
instead of weak-compactness we infer the existence of $\M \in[\N]^\omega$ and
of a weak$^*$-null tree $(z_{\mb}:\mb\in[\M]^{\le k})$ in $X^{**}$ so that
$f(\mb)=z_{\emptyset}+\sum_{i=1}^k z_{\{m_1,\ldots, m_i\}}$,
for all $\mb\in [\M]^k$ and $\|z_{\mb}\| \leq \mathrm{Lip}(f)$, for all
$\mb\in [\M]^{\leq k}\setminus\{\emptyset\}$. For any $\mb \in [\M]^{\leq k}\setminus\{\emptyset\}$ we write
$z_{\mb}=x_{\mb}+e_{\mb}$ with $x_{\mb}\in X$ and $e_{\mb}\in E$.

Fix now $\eta >0$. Since $E$ is finite dimensional, using Ramsey's theorem,
we may assume after further extractions that
\begin{equation}\label{Ramsey1}
\forall i\in \{1,\ldots,k\}\ \  \forall \mb,\nb \in [\M]^{i},
\ \|e_{\mb}-e_{\nb}\|\le \eta.
\end{equation}
It follows from another Ramsey argument that it is enough to construct one pair
$(\mb,\nb)\in I_k(\M)$ such that $\|f(\mb)-f(\nb)\|\le 3(C+1) \Lip(f)$.
We shall build $m_1<n_1<\cdots <m_i<n_i$ inductively, as follows. Since $X$
is $C$-asymptotic $c_0$, we shall play our usual $k$-round game. At each step $i$, the subspace player follows, as she may,
a winning strategy to force the vector player to build a sequence which is
$(C+1)$-equivalent to the canonical basis of $\ell_\infty^k$. So she picks $X_i$ in
${\rm cof}(X)$ according to her winning strategy. Then the vector player picks
$x_i\in S_{X_i}$ and ``we'' choose $m_i<n_i$ in $\M$ according to the following
scheme. The subspace player picks $X_1$ according to her strategy, the
vector player picks $x_1\in S_{X_1}$ and we just pick $m_1<n_1$ in $\M$.
Assume now that $X_1,\dots,X_{i-1}$; $x_1,\ldots,x_{i-1}$ and
$m_1<n_1<\cdots <m_{i-1}<n_{i-1}$ have been chosen for $2\le i \le k$.
For $n>n_{i-1}$, denote $y_n=x_{\{m_1,\ldots,m_{i-1},n\}}-
x_{\{n_1,\ldots,n_{i-1},n+1\}}$ and $v_n=z_{\{m_1,\ldots,m_{i-1},n\}}-
z_{\{n_1,\ldots,n_{i-1},n+1\}}$. The space player picks $X_i \in {\rm cof}(X)$
according to her strategy. Note that $X_i^\perp$ is a finite dimensional weak$^*$
closed subspace of $X^*$.

(a) Assume first that $\liminf_{n\to \infty}\|y_n\|\le \frac{1}{4k} \Lip(f)$.\\
Then we pick $n>n_{i-1}$ such that $\|y_n\|\le \frac{1}{2k} \Lip(f)$, the vector
player picks any $x_i \in S_{X_i}$ and we set $m_i=n$ and $n_i=n+1$.

(b) Assume now that $\liminf_{n\to \infty}\|y_n\|> \frac{1}{4k} \Lip(f)$.\\
Since $(v_n)$ is weak$^*$-null, we have that $(v_n)$ tends uniformly to $0$ on
bounded subsets of $X_i^\perp$. It follows from (\ref{Ramsey1}) and the standard
identification of $(X/X_i)^*$ with $X_i^\perp$ that
$\limsup_{n\to \infty} d(y_n,X_i)\le \eta$. So we can pick $n>n_{i-1}$ such that
$\|y_n\|>\frac{1}{4k} \Lip(f)$ and $d(y_n,X_i)\le 2\eta$, which implies the
existence of $x_i \in S_{X_i}$ so that $\big\| \frac{y_n}{\|y_n\|} -x_i\big\|\le
\frac{16k \eta}{\Lip(f)}$. We set $m_i=n$ and $n_i=n+1$.

This concludes the description of our procedure and we recall that it ensures that
$(x_i)_{i=1}^k$ is $(C+1)$-equivalent to the canonical basis of $\ell_\infty^k$.
We now denote $A$ the set of $i$'s such that procedure (a) has been followed and $B$
the complement of $A$. For simplicity, denote $u_i=x_{\{m_1,\ldots,m_{i}\}}-
x_{\{n_1,\ldots,n_{i}\}}$. We clearly have

\begin{equation*}
\|\sum_{i\in A} u_i \|\le \frac12 \Lip(f).
\end{equation*}
On the other hand, we have
\begin{align*}
\|\sum_{i\in B} u_i \|&\le \big\|\sum_{i\in B} \|u_i\|x_i \big\|+
\big\|\sum_{i\in B} \|u_i\|\big(x_i-\frac{u_i}{\|u_i\|}\big) \big\|\\
&\le (C+1) \max_{i\in B}\|u_i\| +k \max_{i\in B}\|u_i\|\frac{16 k \eta}{\Lip(f)}\\
&\le \big(2 Lip(f)+\eta\big)\big(C+1+\frac{16 k^2 \eta}{\Lip(f)}\big).
\end{align*}
Note that, since $f$ takes values in $X$, we also have that
$f(\mb)-f(\nb)= \sum_{i=1}^k u_i$. Then, combining the above estimates  with an initial choice of a small enough $\eta$, we get that $\|f(\mb)-f(\nb)\|\le 3(C+1)\Lip(f)$.

\end{proof}

We deduce the following.

\begin{cor} There exists a Banach space $X$ which is not reflexive, but such that
the family $([\N]^k,d^{(k)}_\H)_{k\in \N}$ does not
equi-coarsely embed into $X$.
\end{cor}

\begin{proof} We only need to give an example of a quasi-reflexive, but not
reflexive, asymptotic-$c_0$ Banach space. It is based on a construction due to Bellenot, Haydon and Odell \cite{BHO}. For a given Schauder basis $(u_i)$ of a Banach space $X$, the
space $J[(u_i)]$ is defined to be the completion of $c_{00}$ (the space of
finitely supported sequences $(a_i)_{i=1}^\infty$ of real numbers) under the norm
$$\big\|\sum a_ie_i\big\|= \sup\Big\{\Big\|\sum_{i=1}^n \Big(\sum_{j\in s_i}a_j\Big) u_{p_i}\Big\|_X,\ n\in \N, s_1<\cdots <s_n,
\ \min s_i=p_i\Big\},$$
where $s_1,\ldots,s_n$ are intervals in $\N$ and $(e_i)_{i=1}^\infty$ is the canonical basis of $c_{00}$.

It is proved in \cite{BHO} that, if $(u_i)$ is the basis of a reflexive space,
then $J[(u_i)]$ is quasi-reflexive of order one. Let now $(u_i)$ be the unit
vector basis of $T^*$ (see the description of $T^*$ in section \ref{S:3}).
Since $T^*$ is reflexive, $J[(u_i)]$ is quasi-reflexive
of order one and therefore not reflexive. This particular space was first considered in \cite{FOSZ2017} and  estimates {similar} to those given in the proof of Proposition 3.2 in \cite{FOSZ2017} show that $J[(u_i)]$
is asymptotic-$c_0$.
\end{proof}

\bibliographystyle{amsplain}

\end{document}